\documentclass[11pt,a4paper]{article}
\usepackage[T1]{fontenc}
\usepackage[latin1]{inputenc}
\usepackage{amsmath}
\usepackage{amssymb}
\usepackage{color}
\usepackage{fancyhdr}
\usepackage[top=28mm,bottom=33mm,left=21mm,right=21mm]{geometry}
\usepackage{paralist}
\usepackage{amsthm}
\usepackage{mathrsfs}
\usepackage{yfonts}
\usepackage{framed}
\usepackage{pifont}
\usepackage{textcomp}
\usepackage{array}
\usepackage{graphicx}
\usepackage{wrapfig}
\usepackage{enumitem}
\usepackage{hyperref}
\usepackage{pstricks}
\usepackage{pst-plot}

\title{Critical heights of destruction for a forest-fire model on the half-plane}
\author{Robert Graf \\ \small \emph{Mathematisches Institut, Ludwig-Maximilians-Universit\"at M\"unchen} \\ \small \emph{Theresienstr. 39, 80333 M\"unchen, Germany} \\ \small \texttt{robert.graf@math.lmu.de}}
\date{}

\theoremstyle{plain}
\newtheorem{thm}{Theorem}

\newtheorem{cor}{Corollary}

\theoremstyle{definition}
\newtheorem{defn}{Definition}

\newcommand{\mb}{\mathbb}
\newcommand{\mc}{\mathcal}
\newcommand{\mf}{\mathbf}
\newcommand{\dist}{\operatorname{dist}}

\renewcommand{\Im}{\operatorname{Im}}
\newcommand{\pparagraph}[1]{\paragraph[#1]{\normalfont\bfseries #1}}
\newcommand{\keywords}{\textbf{Key words. }\medskip}
\newcommand{\subjclass}{\textbf{MSC 2010. }\medskip}

\begin{document}

\maketitle

\begin{abstract}
Consider the following forest-fire model on the upper half-plane of the triangular lattice: Each site can be ``vacant'' or ``occupied by a tree''. At time $0$ all sites are vacant. Then the process is governed by the following random dynamics: Trees grow at rate $1$, independently for all sites. If an occupied cluster reaches the boundary of the upper half-plane, the cluster is instantaneously destroyed, i.e.\ all of its sites turn vacant. At the critical time $t_c := \log 2$ the process is stopped. Now choose an arbitrary infinite cone in the half-plane whose apex lies on the boundary of the half-plane and whose boundary lines are non-horizontal. We prove that in the final configuration a.s.\ only finitely many sites in the cone have been affected by destruction.
\end{abstract}

\keywords{forest-fire model, half-plane, self-organized criticality}

\subjclass{Primary 60K35, 82C22; Secondary 82B43}

\section{Introduction and statement of the main result}

Forest-fire processes were first introduced by B.\ Drossel and F.\ Schwabl in \cite{drossel1992forest} as a toy model for self-organized criticality. Since then, different variants of forest-fire models have been studied both in the physics and mathematics literature. In \cite{graf2014half}, a forest-fire model on the half-plane was obtained as a subsequential limit of forest-fire models on finite size boxes and a corresponding collection of time- and space-dependent random variables, the so-called heights of destruction, were analysed. It was proven that the heights of destruction in semi-infinite tubes show a phase transition in the sense that they are a.s.\ finite \emph{before} a certain critical time and infinite \emph{after} the critical time. In this paper we show that the heights of destruction in semi-infinite tubes and even in infinite cones are a.s.\ finite \emph{at} the critical time. Since the proof requires two critical exponents of site percolation (equations (\ref{eq correlation length}) and (\ref{eq half-plane one-arm})), which are currently only known for the triangular lattice, we formulate the model on the triangular lattice, wheras \cite{graf2014half} uses the square lattice. In the present section, we give a self-contained account of our result (Theorem~\ref{thm critical height of destruction}), in Section~\ref{sec extension of the model} we put this result in the context of \cite{graf2014half}, and in Section~\ref{sec proof critical height} we give the proof of Theorem~\ref{thm critical height of destruction}.

Let $i = \sqrt{-1}$ denote the imaginary unit, let
\begin{align*}
\mb{T} := \left\{ k + le^{i \pi/3}: k, l \in \mb{Z} \right\}
\end{align*}
be the set of sites of the triangular lattice, let
\begin{align*}
\mb{C}^{\operatorname{u}} := \left\{ z \in \mb{C}: \Im z \geq 0 \right\}
\end{align*}
be the upper half-plane and let $\mb{T}^{\operatorname{u}} := \mb{T} \cap \mb{C}^{\operatorname{u}}$ be the set of sites of the half-plane triangular lattice (see Figure~\ref{fig half plane and half lines}). Note that according to our definition the relation $\mb{Z} \subset \mb{T}^{\operatorname{u}} \subset \mb{C}^{\operatorname{u}}$ holds, where $\mb{Z}$ can be interpreted as the inner boundary of $\mb{T}^{\operatorname{u}}$ in $\mb{T}$. Two sites $v, w \in \mb{T}$ of the triangular lattice are said to be \textbf{neighbours} if their Euclidean distance is $1$. For a subset $S \subset \mb{T}^{\operatorname{u}}$ of the half-plane triangular lattice, we write
\begin{align*}
\partial S := \left\{ v \in \mb{T}^{\operatorname{u}} \setminus S: \left( \exists w \in S: \text{$v$ and $w$ are neighbours} \right) \right\}
\end{align*}
for the \textbf{outer boundary} of $S$ in $\mb{T}^{\operatorname{u}}$. For a site $x \in \mb{Z}$, for example, we have $\partial \{x\} = \{x + 1, x + e^{i \pi/3}, x + e^{2i \pi/3}, x - 1\}$.

In this paper we will study random configurations $(\alpha_v)_{v \in \mb{T}^{\operatorname{u}}} \in \{0, 1\}^{\mb{T}^{\operatorname{u}}}$ created by a forest-fire model (defined below). In order to introduce some notation, let $V \in \{\mb{T}^{\operatorname{u}}, \mb{T}\}$, let $(\alpha_v)_{v \in V} \in \{0, 1\}^V$ and let $j \in \{0, 1\}$. A \textbf{$j$-path} in $(\alpha_v)_{v \in V}$ from a site $y \in V$ to a site $z \in V$ is a sequence $v_0, v_1, \ldots, v_l$ of distinct sites in $V$ (where $l \in \mb{N}_0$) such that the following holds:
\begin{compactitem}
\item $v_0 = y$, $v_l = z$;
\item $v_{k - 1}$ is a neighbour of $v_k$ for all $k \in \{1, \ldots, l\}$;
\item $\alpha_{v_k} = j$ for all $k \in \{0, \ldots, l\}$.
\end{compactitem}
If $Y, Z \subset V$ are subsets, then a $j$-path in $(\alpha_v)_{v \in V}$ from $Y$ to $Z$ is simply any $j$-path in $(\alpha_v)_{v \in V}$ from a site $y \in Y$ to a site $z \in Z$. Moreover, the \textbf{cluster} of a site $y \in V$ in $(\alpha_v)_{v \in V}$ is the set of all sites $z$ in $V$ such that there exists a $1$-path in $(\alpha_v)_{v \in V}$ from $y$ to $z$. If $\alpha_y = 0$, then the cluster of $y$ in $(\alpha_v)_{v \in V}$ is just the empty set.

Informally, the forest-fire model may be described as follows: Each site can be ``vacant'' (denoted by $0$) or ``occupied by a tree'' (denoted by $1$). At time $0$ all sites are vacant. Then the process is governed by two competing random mechanisms: On the one hand, trees grow according to rate $1$ Poisson processes, independently for all sites. On the other hand, if an occupied cluster reaches the boundary of the upper half-plane, the cluster is instantaneously destroyed, i.e.\ all of its sites turn vacant. At the critical time $t_c := \log 2$ the process is stopped.

We now give the formal definition of the forest-fire model, which is similar to the definitions in \cite{graf2014half} and \cite{duerre2006existence}. Here, if $I \subset \mb{R}$ is a left-open interval and $I \ni t \mapsto f_t \in \mb{R}$ is a function, we write $f_{t^-} := \lim_{s \uparrow t} f_s$ for the left-sided limit at $t$, provided the limit exists.

\begin{defn} \label{defn forest-fire Tu}
Let $(\eta_{t, z}, G_{t, z})_{t \in [0, t_c], z \in \mb{T}^{\operatorname{u}}}$ be a process with values in $(\{0, 1\} \times \mb{N}_0)^{[0, t_c] \times \mb{T}^{\operatorname{u}}}$, initial condition $\eta_{0, z} = 0$ for $z \in \mb{T}^{\operatorname{u}}$ and boundary condition $\eta_{t, x} = 0$ for $t \in [0, t_c], x \in \mb{Z}$. Suppose that for all $z \in \mb{T}^{\operatorname{u}}$ the process $(\eta_{t, z}, G_{t, z})_{t \in [0, t_c]}$ is càdlàg. For $z \in \mb{T}^{\operatorname{u}}$ and $t \in (0, t_c]$, let $C_{t^-, z}$ denote the cluster of $z$ in the configuration $(\eta_{t^-, w})_{w \in \mb{T}^{\operatorname{u}}}$.

Then $(\eta_{t, z}, G_{t, z})_{t \in [0, t_c], z \in \mb{T}^{\operatorname{u}}}$ is called a \textbf{$\mb{T}^{\operatorname{u}}$-forest-fire process} if the following conditions are satisfied:
\begin{description}[leftmargin=3.3cm,style=sameline,font=\normalfont]
\item[\text{[POISSON]}] The processes $(G_{t, z})_{t \in [0, t_c]}$, $z \in \mb{T}^{\operatorname{u}}$, are independent Poisson processes with rate~$1$.
\item[\text{[GROWTH]}] For all $t \in (0, t_c]$ and all $z \in \mb{T}^{\operatorname{u}} \setminus \mb{Z}$ the following implications hold:
\begin{compactenum}[(i)]
\item $G_{t^-, z} < G_{t, z} \Rightarrow \eta_{t, z} = 1$, \\
i.e.\ the growth of a tree at the site $z$ at time $t$ implies that the site $z$ is occupied at time $t$;
\item $\eta_{t^-, z} < \eta_{t, z} \Rightarrow G_{t^-, z} < G_{t, z}$, \\
i.e.\ if the site $z$ gets occupied at time $t$, there must have been the growth of a tree at the site $z$ at time $t$.
\end{compactenum}
\item[\text{[DESTRUCTION]}] For all $t \in (0, t_c]$ and all $x \in \mb{Z}$, $z \in \mb{T}^{\operatorname{u}} \setminus \mb{Z}$ the following implications hold:
\begin{compactenum}[(i)]
\item $G_{t^-, x} < G_{t, x} \Rightarrow \forall v \in \partial \{x\} \, \forall w \in C_{t^-, v}: \eta_{t, w} = 0$, \\
i.e.\ if a cluster grows to the boundary $\mb{Z}$ at time $t$, it is destroyed at time $t$;
\item $\eta_{t^-, z} > \eta_{t, z} \Rightarrow \exists u \in \partial C_{t^-, z} \cap \mb{Z}: G_{t^-, u} < G_{t, u}$, \\
i.e.\ if a site is destroyed at time $t$, its cluster must have grown to the boundary $\mb{Z}$ at time $t$.
\end{compactenum}
\end{description}
\end{defn}

In order to construct a $\mb{T}^{\operatorname{u}}$-forest-fire process, we start with independent rate $1$ Poisson processes $(G_{t, z})_{t \in [0, t_c]}$, $z \in \mb{T}^{\operatorname{u}}$, on a probability space $(\Omega, \mc{F}, \mf{P})$, where $\mc{F}$ is the completion of the $\sigma$-field generated by $(G_{t, z})_{t \in [0, t_c], z \in \mb{T}^{\operatorname{u}}}$. We first consider the corresponding \textbf{pure growth process}
\begin{align*}
\sigma_{t, z} := 1_{\{G_{t, z} > 0\}} \text{,} \qquad t \in [0, t_c], z \in \mb{T}^{\operatorname{u}} \text{,}
\end{align*}
on $\mb{T}^{\operatorname{u}}$, where $1_A$ denotes the indicator function of an event $A$. For a fixed time $t \in [0, t_c]$, the configuration $\sigma^{\operatorname{u}}_t := (\sigma_{t, z})_{z \in \mb{T}^{\operatorname{u}}}$ is simply independent site percolation on $\mb{T}^{\operatorname{u}}$, where each site is occupied with probability $1 - e^{-t}$. From the RSW theory we know that at the critical time $t_c$ (where sites are occupied with probability $1/2$), for all $x \in \mb{R}$ we have
\begin{align}
\mf{P} \left[ \text{$\sigma^{\operatorname{u}}_{t_c}$ contains infinitely many disjoint $0$-paths from $\mb{Z}_{<x}$ to $\mb{Z}_{>x}$} \right] &= 1 \text{,} \label{eq RSW 0-path} \\
\mf{P} \left[ \text{$\sigma^{\operatorname{u}}_{t_c}$ contains infinitely many disjoint $1$-paths from $\mb{Z}_{<x}$ to $\mb{Z}_{>x}$} \right] &= 1 \text{,} \label{eq RSW 1-path}
\end{align}
where $\mb{Z}_{<x} := \{x' \in \mb{Z}: x' < x\}$ and $\mb{Z}_{>x} := \{x' \in \mb{Z}: x' > x\}$.
Moreover, it is clear that \emph{if} a $\mb{T}^{\operatorname{u}}$-forest-fire process $(\eta_{t, z}, G_{t, z})_{t \in [0, t_c], z \in \mb{T}^{\operatorname{u}}}$ can be constructed from $(G_{t, z})_{t \in [0, t_c], z \in \mb{T}^{\operatorname{u}}}$, then $(\sigma_{t, z})_{t \in [0, t_c], z \in \mb{T}^{\operatorname{u}}}$ dominates $(\eta_{t, z})_{t \in [0, t_c], z \in \mb{T}^{\operatorname{u}}}$ in the sense that
\begin{align} \label{eq domination pure growth process}
\eta_{s, z} \leq \sigma_{s, z} \leq \sigma_{t, z} \text{,} \qquad 0 \leq s \leq t \leq t_c, z \in \mb{T}^{\operatorname{u}} \text{.}
\end{align}
Equations (\ref{eq RSW 0-path}) and (\ref{eq domination pure growth process}) imply that given $(G_{t, z})_{t \in [0, t_c], z \in \mb{T}^{\operatorname{u}}}$, there exists a unique corresponding $\mb{T}^{\operatorname{u}}$-forest-fire process $(\eta_{t, z}, G_{t, z})_{t \in [0, t_c], z \in \mb{T}^{\operatorname{u}}}$, which can be obtained by partitioning $\mb{T}^{\operatorname{u}}$ into a random collection of finite sets separated by $0$-paths in $\sigma^{\operatorname{u}}_{t_c}$ and performing a graphical construction on each of these sets. (Since (\ref{eq RSW 0-path}) is only an a.s.-property, we may have to change $(G_{t, z})_{t \in [0, t_c], z \in \mb{T}^{\operatorname{u}}}$ on a null set to enable the described partitioning of $\mb{T}^{\operatorname{u}}$ everywhere on $\Omega$.) More details on graphical constructions of interacting particle systems can be found in the book \cite{liggett2004interacting} by T.\ Liggett or the paper \cite{harris1972nearest} by T.\ Harris, who was the first to apply this method.

In this paper we analyse the total effect of destruction in the $\mb{T}^{\operatorname{u}}$-forest-fire process up to the final time $t_c$, which is quantified by the so-called heights of destruction:

\begin{defn}
For $t \in [0, t_c]$ and $S \subset \mb{C}^{\operatorname{u}}$, let
\begin{align} \label{eq defn critical height of destruction}
Y_t(S) := \sup \left\{ \Im z: z \in S \cap \mb{T}^{\operatorname{u}}, \exists s \in (0, t]: \eta_{s^-, z} > \eta_{s, z} \right\} \vee 0
\end{align}
be the height up to which sites in $S$ have been destroyed up to time $t$ (where $Y_t(S)$ can take values in $[0, \infty]$). We call $Y_t(S)$ the \textbf{height of destruction} in $S$ up to time $t$.
\end{defn}

Note that $Y_t(S)$ is monotone increasing in $t$ and $S$ in the sense that for $t_1, t_2 \in [0, t_c]$ and $S_1, S_2 \subset \mb{C}^{\operatorname{u}}$ the implication
\begin{align} \label{eq monotonicity height of destruction}
\left( t_1 \leq t_2 \, \wedge \, S_1 \subset S_2 \right) \Rightarrow Y_{t_1}(S_1) \leq Y_{t_2}(S_2)
\end{align}
holds. We will study the height of destruction in cones of the following kind:

\begin{defn}
For $x \in \mb{R}$ and $\varphi \in (0, \pi/2)$, let
\begin{align*}
K^{\varphi}_{x} := \left\{ x + ae^{i \varphi} + be^{i (\pi - \varphi)}: a, b \geq 0 \right\}
\end{align*}
denote the infinite cone whose apex is $x$ and whose boundary lines have angular directions $\varphi$ and $\pi - \varphi$, respectively (see Figure~\ref{fig half plane and half lines}).
\end{defn}

\begin{figure}
\centering
\psset{unit=0.8cm}
\begin{pspicture*}(-0.9,-0.6)(15.9,8.2)
\newgray{verylightgray}{0.85}
\pspolygon[linestyle=none,fillstyle=solid,fillcolor=verylightgray](3.6,7.428)(2,0)(0.4,7.428)
\pspolygon[linestyle=none,fillstyle=solid,fillcolor=verylightgray](4.25,7.428)(4.25,0)(5.25,0)(5.25,7.428)
\pspolygon[linestyle=none,fillstyle=solid,fillcolor=verylightgray](5.789,7.428)(7.089,0)(8.111,0)(6.811,7.428)
\pspolygon[linestyle=none,fillstyle=solid,fillcolor=verylightgray](13.712,7.428)(9.423,0)(10.577,0)(14.866,7.428)
\multido{\ix=0+1}{16}{
\multido{\ry=0+1.733}{5}{
\psdots[dotscale=1,dotstyle=*,fillstyle=solid,linecolor=gray](\ix,\ry)
}}
\multido{\rx=0.5+1}{15}{
\multido{\ry=0.866+1.733}{4}{
\psdots[dotscale=1,dotstyle=*,fillstyle=solid,linecolor=gray](\rx,\ry)
}}
\multido{\ry=0+0.866}{9}{
\psline[linewidth=0.03,linestyle=solid,linecolor=gray](-0.25,\ry)(15.25,\ry)
}
\multido{\ixd=0+1,\rxu=4.144+1}{12}{
\psline[linewidth=0.03,linestyle=solid,linecolor=gray](\ixd,0)(\rxu,7.178)
}
\psline[linewidth=0.03,linestyle=solid,linecolor=gray](-0.25,6.495)(0.144,7.178)
\psline[linewidth=0.03,linestyle=solid,linecolor=gray](-0.25,4.763)(1.144,7.178)
\psline[linewidth=0.03,linestyle=solid,linecolor=gray](-0.25,3.031)(2.144,7.178)
\psline[linewidth=0.03,linestyle=solid,linecolor=gray](-0.25,1.299)(3.144,7.178)
\psline[linewidth=0.03,linestyle=solid,linecolor=gray](12,0)(15.25,5.629)
\psline[linewidth=0.03,linestyle=solid,linecolor=gray](13,0)(15.25,3.897)
\psline[linewidth=0.03,linestyle=solid,linecolor=gray](14,0)(15.25,2.165)
\psline[linewidth=0.03,linestyle=solid,linecolor=gray](15,0)(15.25,0.433)
\multido{\ixd=4+1,\rxu=-0.144+1}{12}{
\psline[linewidth=0.03,linestyle=solid,linecolor=gray](\ixd,0)(\rxu,7.178)
}
\psline[linewidth=0.03,linestyle=solid,linecolor=gray](0,0)(-0.25,0.433)
\psline[linewidth=0.03,linestyle=solid,linecolor=gray](1,0)(-0.25,2.165)
\psline[linewidth=0.03,linestyle=solid,linecolor=gray](2,0)(-0.25,3.897)
\psline[linewidth=0.03,linestyle=solid,linecolor=gray](3,0)(-0.25,5.629)
\psline[linewidth=0.03,linestyle=solid,linecolor=gray](15.25,1.299)(11.856,7.178)
\psline[linewidth=0.03,linestyle=solid,linecolor=gray](15.25,3.031)(12.856,7.178)
\psline[linewidth=0.03,linestyle=solid,linecolor=gray](15.25,4.763)(13.856,7.178)
\psline[linewidth=0.03,linestyle=solid,linecolor=gray](15.25,6.495)(14.856,7.178)
\psline[linewidth=0.06,linestyle=solid,linecolor=black,fillstyle=none](0.4,7.428)(2,0)(3.6,7.428)
\psline[linewidth=0.06,linestyle=solid,linecolor=black,fillstyle=none](4.25,7.428)(4.25,0)(5.25,0)(5.25,7.428)
\psline[linewidth=0.06,linestyle=solid,linecolor=black,fillstyle=none](5.789,7.428)(7.089,0)(8.111,0)(6.811,7.428)
\psline[linewidth=0.06,linestyle=solid,linecolor=black,fillstyle=none](13.712,7.428)(9.423,0)(10.577,0)(14.866,7.428)
\psline[linewidth=0.03,linestyle=dashed,linecolor=black](2,0)(2,7.428)
\psline[linewidth=0.03,linestyle=dashed,linecolor=black](4.75,0)(4.75,7.428)
\psline[linewidth=0.03,linestyle=dashed,linecolor=black](7.6,0)(6.3,7.428)
\psline[linewidth=0.03,linestyle=dashed,linecolor=black](10,0)(14.289,7.428)
\psarc[linewidth=0.03,linestyle=solid,linecolor=black](2,0){0.75}{0}{77.8}
\psarc[linewidth=0.03,linestyle=solid,linecolor=black](5.25,0){0.75}{0}{90}
\psarc[linewidth=0.03,linestyle=solid,linecolor=black](8.111,0){0.75}{0}{99.9}
\psarc[linewidth=0.03,linestyle=solid,linecolor=black](10.577,0){0.75}{0}{60}
\psdots[dotscale=1.5,dotstyle=+,dotangle=45,fillstyle=solid,linecolor=black](2,0)
\psdots[dotscale=1.5,dotstyle=+,dotangle=45,fillstyle=solid,linecolor=black](4.75,0)
\psdots[dotscale=1.5,dotstyle=+,dotangle=45,fillstyle=solid,linecolor=black](7.6,0)
\psdots[dotscale=1.5,dotstyle=+,dotangle=45,fillstyle=solid,linecolor=black](10,0)
\rput[lt](1.85,-0.3){$x$}
\rput[lt](4.55,-0.3){$x_1$}
\rput[lt](7.4,-0.3){$x_2$}
\rput[lt](9.8,-0.3){$x_3$}
\psline[linewidth=0.02,linestyle=dashed,dash=2pt 1pt,linecolor=black]{<-}(2.292,0.236)(3,0.8)
\rput[lt](2.85,1.25){$\varphi$}
\psline[linewidth=0.02,linestyle=dashed,dash=2pt 1pt,linecolor=black]{<-}(5.515,0.265)(6,0.8)
\rput[lt](5.8,1.25){$\varphi_1$}
\psline[linewidth=0.02,linestyle=dashed,dash=2pt 1pt,linecolor=black]{<-}(8.398,0.241)(9,0.8)
\rput[lt](8.8,1.25){$\varphi_2$}
\psline[linewidth=0.02,linestyle=dashed,dash=2pt 1pt,linecolor=black]{<-}(10.902,0.1875)(12,0.8)
\rput[lt](11.8,1.25){$\varphi_3$}
\rput[lt](1.8,8){$K^{\varphi}_x$}
\rput[lt](4.55,8){$M^{\varphi_1}_{x_1}$}
\rput[lt](6.1,8){$M^{\varphi_2}_{x_2}$}
\rput[lt](14.089,8){$M^{\varphi_3}_{x_3}$}
\rput[lt](15.3,6.2){$\mb{T}^{\operatorname{u}}$}
\end{pspicture*}
\caption{The half-plane triangular lattice $\mb{T}^{\operatorname{u}}$, a cone $K^{\varphi}_x$ and three semi-infinite tubes $M^{\varphi_j}_{x_j}$ ($j = 1, 2, 3$).}
\label{fig half plane and half lines}
\end{figure}
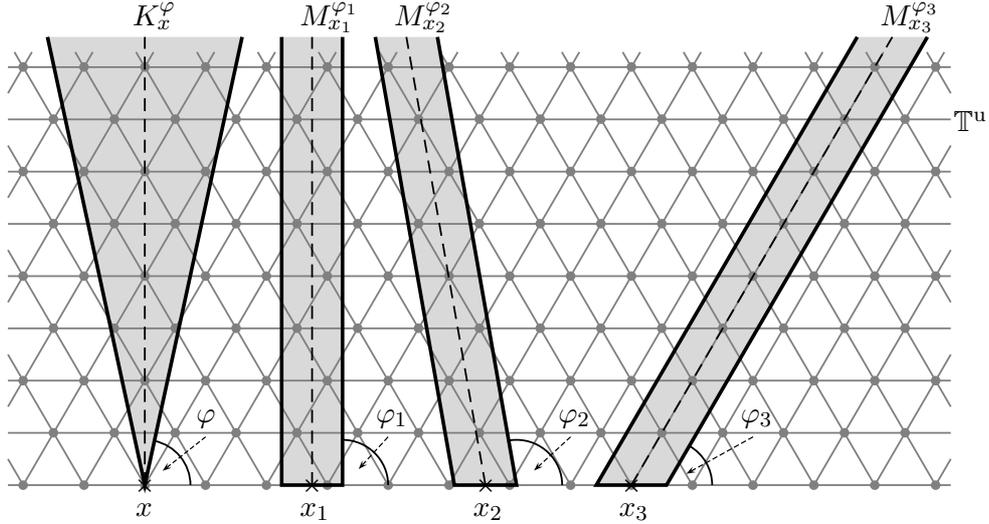

Equation (\ref{eq RSW 1-path}) indicates that $Y_{t_c}(K^{\varphi}_x)$ could potentially be equal to $\infty$. We prove that this case a.s.\ does not occur:

\begin{thm} \label{thm critical height of destruction}
For all $x \in \mb{R}$ and $\varphi \in (0, \pi/2)$ we have $\mf{P} \left[ Y_{t_c}(K^{\varphi}_x) < \infty \right] = 1$.
\end{thm}

Roughly speaking, Theorem~\ref{thm critical height of destruction} means that up to the final time $t_c$, the influence of the destruction mechanism [DESTRUCTION] in Definition~\ref{defn forest-fire Tu} is confined to areas close to the inner boundary $\mb{Z}$ of the half-plane lattice $\mb{T}^{\operatorname{u}}$.

\section{Extension of the model beyond the critical time} \label{sec extension of the model}

It is a natural question to ask how the forest-fire model and the corresponding heights of destruction behave when we let the process run beyond the critical time $t_c$. In this case the local graphical construction above does not work any more so that we must first give thought to the existence of such an extended process. In fact, it is not known whether a process $(\eta_{t, z}, G_{t, z})_{t \in [0, \infty), z \in \mb{T}^{\operatorname{u}}}$ satisfying Definition~\ref{defn forest-fire Tu} for all $t \in [0, \infty)$ exists. However, if we additionally demand that clusters are also destroyed when they are about to become infinite, then the extension does exist. This motivates the following definition:

\begin{defn} \label{defn extended forest-fire Tu}
Let $(\eta_{t, z}, G_{t, z})_{t \in [0, \infty), z \in \mb{T}^{\operatorname{u}}}$ be a process with values in $(\{0, 1\} \times \mb{N}_0)^{[0, \infty) \times \mb{T}^{\operatorname{u}}}$, initial condition $\eta_{0, z} = 0$ for $z \in \mb{T}^{\operatorname{u}}$ and boundary condition $\eta_{t, x} = 0$ for $t \in [0, \infty), x \in \mb{Z}$. Suppose that for all $z \in \mb{T}^{\operatorname{u}}$ the process $(\eta_{t, z}, G_{t, z})_{t \in [0, \infty)}$ is càdlàg. For $z \in \mb{T}^{\operatorname{u}}$ and $t \in (0, \infty)$, let $C_{t^-, z}$ denote the cluster of $z$ in the configuration $(\eta_{t^-, w})_{w \in \mb{T}^{\operatorname{u}}}$.

Then $(\eta_{t, z}, G_{t, z})_{t \in [0, \infty), z \in \mb{T}^{\operatorname{u}}}$ is called an \textbf{extended $\mb{T}^{\operatorname{u}}$-forest-fire process} if the following conditions are satisfied:
\begin{description}[leftmargin=3.3cm,style=sameline,font=\normalfont]
\item[\text{[POISSON]}] The processes $(G_{t, z})_{t \in [0, \infty)}$, $z \in \mb{T}^{\operatorname{u}}$, are independent Poisson processes with rate~$1$.
\item[\text{[TRANSL-INV]}] The distribution of $(\eta_{t, z}, G_{t, z})_{t \in [0, \infty), z \in \mb{T}^{\operatorname{u}}}$ is invariant under translations along the real line, i.e.\ the processes $(\eta_{t, z}, G_{t, z})_{t \in [0, \infty), z \in \mb{T}^{\operatorname{u}}}$ and $(\eta_{t, z + 1}, G_{t, z + 1})_{t \in [0, \infty), z \in \mb{T}^{\operatorname{u}}}$ have the same distribution.
\item[\text{[GROWTH]}] For all $t \in (0, \infty)$ and all $z \in \mb{T}^{\operatorname{u}} \setminus \mb{Z}$ the following implications hold:
\begin{compactenum}[(i)]
\item $G_{t^-, z} < G_{t, z} \Rightarrow \eta_{t, z} = 1$, \\
i.e.\ the growth of a tree at the site $z$ at time $t$ implies that the site $z$ is occupied at time $t$;
\item $\eta_{t^-, z} < \eta_{t, z} \Rightarrow G_{t^-, z} < G_{t, z}$, \\
i.e.\ if the site $z$ gets occupied at time $t$, there must have been the growth of a tree at the site $z$ at time $t$.
\end{compactenum}
\item[\text{[DESTRUCTION]}] For all $t \in (0, \infty)$ and all $x \in \mb{Z}$, $z \in \mb{T}^{\operatorname{u}} \setminus \mb{Z}$ the following implications hold:
\begin{compactenum}[(i)]
\item $\left(G_{t^-, x} < G_{t, x} \Rightarrow \forall v \in \partial \{x\} \,\forall w \in C_{t^-, v}: \eta_{t, w} = 0 \right) \wedge$ \\
$\left( |C_{t^-, z}| = \infty \Rightarrow \forall w \in C_{t^-, z}: \eta_{t, w} = 0 \right)$, \\
i.e.\ if a cluster grows to the boundary $\mb{Z}$ at time $t$, it is destroyed at time $t$, and if a cluster is about to become infinite at time $t$, it is destroyed at time~$t$;
\item $\eta_{t^-, z} > \eta_{t, z} \Rightarrow \left( \left( \exists u \in \partial C_{t^-, z} \cap \mb{Z}: G_{t^-, u} < G_{t, u} \right) \vee |C_{t^-, z}| = \infty \right)$, \\
i.e.\ if a site is destroyed at time $t$, its cluster either must have grown to the boundary $\mb{Z}$ at time $t$ or it must have been about to become infinite at time~$t$.
\end{compactenum}
\end{description}
\end{defn}

The existence of an extended $\mb{T}^{\operatorname{u}}$-forest-fire process follows from Theorem~1.2 in \cite{graf2014half}. In this reference, an analogous process on the upper half-plane of the square lattice $\mb{Z}^2$ is shown to exist and the proof can be directly transferred to the triangular lattice. Conversely, it is currently not known whether extended $\mb{T}^{\operatorname{u}}$-forest-fire processes are unique in distribution. This is the reason why we have included the translation-invariance property [TRANSL-INV] in Definition~\ref{defn extended forest-fire Tu}, wheras for the unextended $\mb{T}^{\operatorname{u}}$-forest-fire process, the translation-invariance is already implied by the uniqueness of this process. Since extended $\mb{T}^{\operatorname{u}}$-forest-fire processes are also dominated by the corresponding pure growth process, in which there are a.s.\ no infinite clusters until the critical time $t_c$, the destruction of infinite clusters in extended $\mb{T}^{\operatorname{u}}$-forest-fire processes a.s.\ does not occur until $t_c$, i.e.\ [DESTRUCTION] in Definition~\ref{defn extended forest-fire Tu} and [DESTRUCTION] in Definition~\ref{defn forest-fire Tu} a.s.\ coincide until $t_c$. (In fact, it is unclear whether the destruction of infinite clusters in extended $\mb{T}^{\operatorname{u}}$-forest-fire processes ever occurs with positive probability.) Hence, if $(\eta_{t, z}, G_{t, z})_{t \in [0, \infty), z \in \mb{T}^{\operatorname{u}}}$ is an extended $\mb{T}^{\operatorname{u}}$-forest-fire process, then restricted to the complement of a null set, $(\eta_{t, z}, G_{t, z})_{t \in [0, t_c], z \in \mb{T}^{\operatorname{u}}}$ is a $\mb{T}^{\operatorname{u}}$-forest-fire process.

For the remainder of this section, let $(\eta_{t, z}, G_{t, z})_{t \in [0, \infty), z \in \mb{T}^{\operatorname{u}}}$ be an extended $\mb{T}^{\operatorname{u}}$-forest-fire process on a probability space $(\Omega, \mc{F}, \mf{P})$, where $\mc{F}$ is the completion of the $\sigma$-field generated by $(\eta_{t, z}, G_{t, z})_{t \in [0, \infty), z \in \mb{T}^{\operatorname{u}}}$. For $t \in [0, \infty)$ and $S \subset \mb{C}^{\operatorname{u}}$, we define the corresponding height of destruction $Y_t(S)$ in $S$ up to time $t$ as in equation (\ref{eq defn critical height of destruction}). Moreover, for $z \in \mb{C}$ and $S \subset \mb{C}$, we define the distance between $z$ and $S$ by
\begin{align} \label{eq dist}
\dist(z, S) := \inf \left\{ |z - z'|: z' \in S \right\} \text{.}
\end{align}
We now look at the height of destruction in semi-infinite tubes of the following kind:

\begin{defn}
For $x \in \mb{R}$ and $\varphi \in (0, \pi)$, let
\begin{align*}
L^{\varphi}_{x} := \left\{ x + y e^{i \varphi}: y \geq 0 \right\}
\end{align*}
denote the half-line with starting point $x$ and angular direction $\varphi$ and let
\begin{align*}
M^{\varphi}_{x} := \left\{ z \in \mb{C}^{\operatorname{u}}: \dist(z, L^{\varphi}_{x}) \leq \frac{1}{2} \right\}
\end{align*}
denote the semi-infinite tube with centre line $L^{\varphi}_{x}$ and width $1$ (see Figure~\ref{fig half plane and half lines}).
\end{defn}

Combining Theorem~\ref{thm critical height of destruction} with results in \cite{graf2014half}, we obtain the following statement:

\begin{cor} \label{cor critical height of destruction}
For $x \in \mb{R}$ and $\varphi \in (0, \pi)$ we have
\begin{align}
\mf{P} \left[ \forall t \in [0, t_c]: Y_t(M^{\varphi}_{x}) < \infty \right] &= 1 \text{,} \label{cor critical height eq subcritical} \\
\mf{P} \left[ \forall t \in (t_c, \infty): Y_t(M^{\varphi}_{x}) = \infty \right] &= 1 \text{.} \label{cor critical height eq supercritical}
\end{align}
\end{cor}

In other words, the height of destruction in the semi-infinite tube $M^{\varphi}_{x}$ shows a phase transition in the sense that it is finite until the critical time $t_c$ and becomes infinite immediately after~$t_c$.

\begin{proof}[Proof of Corollary~\ref{cor critical height of destruction}]
Let $x \in \mb{R}$ and $\varphi \in (0, \pi)$. Pick $\alpha \in (0, \pi/2)$ such that $\alpha < \min \{\varphi, \pi - \varphi\}$ holds. Since $(M^{\varphi}_{x} \setminus K^{\alpha}_{x}) \cap \mb{T}^{\operatorname{u}}$ contains only finitely many sites and since the height of destruction is monotone increasing in the sense of (\ref{eq monotonicity height of destruction}), equation (\ref{cor critical height eq subcritical}) is an immediate consequence of Theorem~\ref{thm critical height of destruction}.

Equation (\ref{cor critical height eq supercritical}) can be proved along the lines of Theorem~1.5 in \cite{graf2014half}. In this reference, a corresponding statement is proved for a slightly different setting, namely for an analogous forest-fire model on the upper half-plane of the square lattice $\mb{Z}^2$ and for $x \in \mb{Z}$, $\varphi = \pi/2$. (The associated height of destruction up to time $t$ is denoted by $Y_{t, x}$ in \cite{graf2014half}.) However, the backbone of the proof in \cite{graf2014half} does not depend on these particular assumptions. A crucial property of $M^{\varphi}_x$ in the course of the proof is the fact that any $1$-path which crosses from the left of $M^{\varphi}_x$ to the right of $M^{\varphi}_x$ has at least one site in $M^{\varphi}_x$; this is the reason why we have defined $M^{\varphi}_x$ to have width $1$.
\end{proof}

\section{Proof of Theorem~\ref{thm critical height of destruction}} \label{sec proof critical height}

Let $x \in \mb{R}$ and $\varphi \in (0, \pi/2)$. Throughout this section, we consider the following setting: Let $(G_{t, z})_{t \in [0, t_c]}$, $z \in \mb{T}$, be independent Poisson processes on a probability space $(\Omega, \mc{F}, \mf{P})$, where $\mc{F}$ is the completion of the $\sigma$-field generated by $(G_{t, z})_{t \in [0, t_c], z \in \mb{T}}$, and let
\begin{align*}
\sigma_{t, z} := 1_{\{G_{t, z} > 0\}} \text{,} \qquad t \in [0, t_c], z \in \mb{T} \text{,}
\end{align*}
be the corresponding pure growth process on $\mb{T}$. (It will be convenient to have these processes on the whole triangular lattice $\mb{T}$ and not just on $\mb{T}^{\operatorname{u}}$.) Moreover, let $(\eta_{t, z}, G_{t, z})_{t \in [0, t_c], z \in \mb{T}^{\operatorname{u}}}$ be the corresponding $\mb{T}^{\operatorname{u}}$-forest-fire process (for the construction of which $(G_{t, z})_{t \in [0, t_c], z \in \mb{T}^{\operatorname{u}}}$ may have to be changed on a null set) and let $Y_{t_c}(K^{\varphi}_x)$ be the associated height of destruction in the cone $K^{\varphi}_x$ up to the critical time $t_c$. For $t \in [0, t_c]$, we henceforth abbreviate $\eta_t := (\eta_{t, z})_{z \in \mb{T}^{\operatorname{u}}}$, $\sigma^{\operatorname{u}}_t := (\sigma_{t, z})_{z \in \mb{T}^{\operatorname{u}}}$ and $\sigma_t := (\sigma_{t, z})_{z \in \mb{T}}$.

We will frequently use the following terminology: Let $V \in \{\mb{T}^{\operatorname{u}}, \mb{T}\}$, let $(\alpha_v)_{v \in V} \in \{0, 1\}^V$ be a random configuration and let $w \in V$, $S \subset \mb{C}$. Then we write $\{w \leftrightarrow S \text{ in } (\alpha_v)_{v \in V}\}$ (in words: $w$ is connected to $S$ in $(\alpha_v)_{v \in V}$) for the event that there exists a $1$-path in $(\alpha_v)_{v \in V}$ from a site $y \in V$ to a site $z \in V$ such that $y$ is a neighbour of $w$ and $\dist(z, S) \leq 1$ holds, where $\dist(z, S)$ is defined as in (\ref{eq dist}). Note that our definition of $\{w \leftrightarrow S \text{ in } (\alpha_v)_{v \in V}\}$ does not impose any condition on the site $w$ itself.

\subsection{Tools from percolation theory}

We will need the following results from percolation theory:

\pparagraph{Exponential decay in the subcritical regime.}
For $z \in \mb{T}$ and $n \in \mb{N}$, let
\begin{align*}
S^{\varphi}_n(z)
:= &\left\{ z + u + ve^{i\varphi}: u, v \in \mb{R}, |u| = n, |v| \leq n \right\} \\
&\cup \left\{ z + u + ve^{i\varphi}: u, v \in \mb{R}, |u| \leq n, |v| = n \right\}
\end{align*}
denote the surface of the rhombus with centre $z$, side length $2n$ and sides parallel to the $\mb{R}$-basis $\{1, e^{i\varphi}\}$ of $\mb{C}$. There exists a function $\xi_{\varphi}: (0, t_c) \rightarrow (0, \infty)$ such that for all $t \in (0, t_c)$ the full-plane one-arm event $\{0 \leftrightarrow S^{\varphi}_n(0) \text{ in } \sigma_t\}$ satisfies 
\begin{align} \label{eq exponential decay limit}
\lim_{n \rightarrow \infty} - \frac{\log \mf{P} \left[ 0 \leftrightarrow S^{\varphi}_n(0) \text{ in } \sigma_t \right]}{n} = \frac{1}{\xi_{\varphi}(t)} \text{;}
\end{align}
$\xi_{\varphi}(t)$ is called the correlation length of the configuration $\sigma_t$. Moreover, there exists a universal constant $c \in (0, \infty)$ such that for all $t \in (0, t_c)$ and $n \in \mb{N}$
\begin{align} \label{eq exponential decay upper bound}
\mf{P} \left[ 0 \leftrightarrow S^{\varphi}_n(0) \text{ in } \sigma_t \right]
\leq cn \exp \left( -\frac{n}{\xi_{\varphi}(t)} \right)
\end{align}
holds. For the proof of (\ref{eq exponential decay limit}) and (\ref{eq exponential decay upper bound}) the reader is referred to \cite{grimmett2010percolation}, Section 6.1. (In this reference, analogous statements are proven for bond percolation on the square lattice $\mb{Z}^2$ and $\varphi = \pi/2$ but the proofs can be transferred one-to-one to our setting.)

\pparagraph{Critical exponents.}
Near the critical time $t_c$, the correlation length behaves like
\begin{align} \label{eq correlation length}
\xi_{\varphi}(t)
= (t_c - t)^{-4/3 + o(1)} \qquad \text{for } t \uparrow t_c \text{.}
\end{align}
At the critical time $t_c$, the probability of the half-plane one-arm event $\{0 \leftrightarrow S^{\varphi}_n(0) \cap \mb{C}^{\operatorname{u}} \text{ in } \sigma^{\operatorname{u}}_{t_c}\}$ decays like
\begin{align} \label{eq half-plane one-arm}
\mf{P} \left[ 0 \leftrightarrow S^{\varphi}_n(0) \cap \mb{C}^{\operatorname{u}} \text{ in } \sigma^{\operatorname{u}}_{t_c} \right]
= n^{-1/3 + o(1)} \qquad \text{for } n \rightarrow \infty \text{.}
\end{align}
Equations (\ref{eq correlation length}) and (\ref{eq half-plane one-arm}) were first proven by S.\ Smirnov and W.\ Werner in \cite{smirnov2001near} (Theorems~1(iv) and 3) and are also discussed in the survey article \cite{nolin2008near} (Theorems~33(i) and 22). (In these references, the statements are not based on the rhombus $S^{\varphi}_n(0)$ used here but on the circle with centre $0$ and radius $n$ and the rhombus $S^{\pi/3}_n(0)$ with angle $\pi/3$, respectively. In fact, the exact shape of the boundary line is irrelevant. However, the current proof of (\ref{eq correlation length}) and (\ref{eq half-plane one-arm}) only works for the triangular lattice.)

\subsection{The core of the proof of Theorem~\ref{thm critical height of destruction}}

We now prove Theorem~\ref{thm critical height of destruction}, i.e.\ we show that $\mf{P}[Y_{t_c}(K^{\varphi}_x) = \infty] = 0$. Since the $\mb{T}^{\operatorname{u}}$-forest-fire process $(\eta_{t, z}, G_{t, z})_{t \in [0, t_c], z \in \mb{T}^{\operatorname{u}}}$ is dominated by the corresponding pure growth process $(\sigma_{t, z})_{t \in [0, t_c], z \in \mb{T}^{\operatorname{u}}}$ in the sense of equation (\ref{eq domination pure growth process}), a.s.\ all destroyed clusters in $(\eta_{t, z}, G_{t, z})_{t \in [0, t_c], z \in \mb{T}^{\operatorname{u}}}$ are finite. Hence, if $Y_{t_c}(K^{\varphi}_x) = \infty$ holds, then a.s.\ infinitely many clusters which reach from $K^{\varphi}_x$ to the inner boundary $\mb{Z}$ must have been destroyed up to the critical time $t_c$. Moreover, since there are only finitely many jumps in a rate $1$ Poisson process up to time $t_c$, every site on the inner boundary $\mb{Z}$ can only be the origin of finitely many destruction events up to time $t_c$. This implies the inclusion
\begin{align*}
\left\{ Y_{t_c}(K^{\varphi}_x) = \infty \right\}
\stackrel{\text{a.s.}}{\subset} \limsup_{n \rightarrow \infty} \mc{A}^{\varphi}_{x, n} \cup \limsup_{n \rightarrow \infty} \mc{A}^{\varphi}_{x, -n} \text{,}
\end{align*}
where we define
\begin{align*}
\mc{A}^{\varphi}_{x, n} &:= \left\{ \exists t \in [0, t_c): \lceil x \rceil + n \leftrightarrow K^{\varphi}_{x} \text{ in } \eta_t, \operatorname{G}_{t, t_c, \lceil x \rceil + n} \right\} \text{,} \\
\mc{A}^{\varphi}_{x, -n} &:= \left\{ \exists t \in [0, t_c): \lfloor x \rfloor - n \leftrightarrow K^{\varphi}_{x} \text{ in } \eta_t, \operatorname{G}_{t, t_c, \lfloor x \rfloor - n} \right\}
\end{align*}
for $n \in \mb{N}$ and use the abbreviation
\begin{align*}
\operatorname{G}_{s, t, z} := \left\{ G_{s, z} < G_{t, z} \right\}
\end{align*}
for $0 \leq s < t \leq t_c$ and $z \in \mb{T}^{\operatorname{u}}$. By symmetry, we have $\mf{P}[\mc{A}^{\varphi}_{x, -n}] = \mf{P}[\mc{A}^{\varphi}_{-x, n}]$ for all $n \in \mb{N}$; consequently, it suffices to prove
\begin{align} \label{thm critical height of destruction eq to prove}
\mf{P} \left[ \limsup_{n \rightarrow \infty} \mc{A}^{\varphi}_{x, n} \right] = 0 \text{.}
\end{align}
Applying equation (\ref{eq domination pure growth process}) once more and using the topological fact that any connection $\lceil x \rceil + n \leftrightarrow K^{\varphi}_{x}$ necessarily contains a connection $\lceil x \rceil + n \leftrightarrow S^{\varphi}_n(\lceil x \rceil + n) \cap \mb{C}^{\operatorname{u}}$, we obtain the inclusions
\begin{align}
\mc{A}^{\varphi}_{x, n}
&\subset \left\{ \exists t \in [0, t_c): \lceil x \rceil + n \leftrightarrow K^{\varphi}_{x} \text{ in } \sigma^{\operatorname{u}}_t, \operatorname{G}_{t, t_c, \lceil x \rceil + n} \right\} \nonumber \\
&\subset \left\{ \exists t \in [0, t_c): \lceil x \rceil + n \leftrightarrow S^{\varphi}_n(\lceil x \rceil + n) \cap \mb{C}^{\operatorname{u}} \text{ in } \sigma^{\operatorname{u}}_t, \operatorname{G}_{t, t_c, \lceil x \rceil + n} \right\}
=: \mc{B}^{\varphi}_{x, n} \text{.} \label{thm critical height of destruction eq A subset B}
\end{align}
Now choose an arbitrary $\delta \in (0, 1/12)$ and consider the event
\begin{align*}
\mc{C}^{\varphi, \delta}_{x, n}
:= \left\{ \exists t \in [0, t_c - n^{-3/4 + \delta}): \lceil x \rceil + n \leftrightarrow S^{\varphi}_n(\lceil x \rceil + n) \cap \mb{C}^{\operatorname{u}} \text{ in } \sigma^{\operatorname{u}}_t \right\}
\end{align*}
that the connection $\lceil x \rceil + n \leftrightarrow S^{\varphi}_n(\lceil x \rceil + n) \cap \mb{C}^{\operatorname{u}}$ in the pure growth process already occurs before time $t_c - n^{-3/4 + \delta}$ (where $n \in \mb{N}$ is assumed to be large enough to guarantee $t_c - n^{-3/4 + \delta} > 0$). We can estimate the probability of this event from above as follows:
\begin{align*}
\mf{P} \left[ \mc{C}^{\varphi, \delta}_{x, n} \right]
&\leq \mf{P} \left[ \exists t \in [0, t_c - n^{-3/4 + \delta}): 0 \leftrightarrow S^{\varphi}_n(0) \text{ in } \sigma_t \right] \\
&= \mf{P} \left[ 0 \leftrightarrow S^{\varphi}_n(0) \text{ in } \sigma_{t_c - n^{-3/4 + \delta}} \right] \\
&\leq cn \exp \left( -\frac{n}{\xi_{\varphi} \left( t_c - n^{-3/4 + \delta} \right)} \right) \\
&= cn \exp \left( -\frac{n}{\left( n^{-3/4 + \delta} \right)^{-4/3 + o(1)}} \right) \qquad \text{for } n \rightarrow \infty \\
&= cn \exp \left( -n^{(4/3) \delta + o(1)} \right) \qquad \text{for } n \rightarrow \infty \text{.}
\end{align*}
Here we first drop the condition that the connection occurs in the upper half-plane $\mb{C}^{\operatorname{u}}$ and use the translation-invariance of the pure growth process; then we employ the fact that $\sigma_t$ is monotone increasing in $t$; finally we successively apply equations (\ref{eq exponential decay upper bound}) and (\ref{eq correlation length}). In particular, this estimate implies
\begin{align*}
\sum_{n = 1}^{\infty} \mf{P} \left[ \mc{C}^{\varphi, \delta}_{x, n} \right] < \infty
\end{align*}
and hence
\begin{align*}
\mf{P} \left[ \limsup_{n \rightarrow \infty} \mc{C}^{\varphi, \delta}_{x, n} \right] = 0
\end{align*}
by the Borel-Cantelli lemma. Regarding the limes superior of the events $\mc{B}^{\varphi}_{x, n}$ (defined in (\ref{thm critical height of destruction eq A subset B})), we thus conclude
\begin{align} \label{thm critical height of destruction eq limsup B}
\limsup_{n \rightarrow \infty} \mc{B}^{\varphi}_{x, n}
\stackrel{\text{a.s.}}{\subset} \limsup_{n \rightarrow \infty} \left( \mc{B}^{\varphi}_{x, n} \setminus \mc{C}^{\varphi, \delta}_{x, n} \right)
\subset \limsup_{n \rightarrow \infty} \mc{D}^{\varphi, \delta}_{x, n} \text{,}
\end{align}
where we abbreviate
\begin{align*}
\mc{D}^{\varphi, \delta}_{x, n}
:= \left\{ \exists t \in [t_c - n^{-3/4 + \delta}, t_c): \lceil x \rceil + n \leftrightarrow S^{\varphi}_n(\lceil x \rceil + n) \cap \mb{C}^{\operatorname{u}} \text{ in } \sigma^{\operatorname{u}}_t, \operatorname{G}_{t, t_c, \lceil x \rceil + n} \right\}
\end{align*}
for $n \in \mb{N}$ satisfying $t_c - n^{-3/4 + \delta} > 0$. The probability of the event $\mc{D}^{\varphi, \delta}_{x, n}$ can be bounded from above as follows:
\begin{align*}
\mf{P} \left[ \mc{D}^{\varphi, \delta}_{x, n} \right]
&\leq \mf{P} \left[ \lceil x \rceil + n \leftrightarrow S^{\varphi}_n(\lceil x \rceil + n) \cap \mb{C}^{\operatorname{u}} \text{ in } \sigma^{\operatorname{u}}_{t_c}, \operatorname{G}_{t_c - n^{-3/4 + \delta}, t_c, \lceil x \rceil + n} \right] \\
&= \mf{P} \left[ 0 \leftrightarrow S^{\varphi}_n(0) \cap \mb{C}^{\operatorname{u}} \text{ in } \sigma^{\operatorname{u}}_{t_c} \right] \, \mf{P} \left[ \operatorname{G}_{0, n^{-3/4 + \delta}, 0} \right] \\
&= n^{-1/3 + o(1)} \cdot \left( 1 - \exp \left( -n^{-3/4 + \delta} \right) \right) \qquad \text{for } n \rightarrow \infty \\
&\leq n^{-1/3 + o(1)} \cdot n^{-3/4 + \delta} \qquad \text{for } n \rightarrow \infty \\
&= n^{-13/12 + \delta + o(1)} \qquad \text{for } n \rightarrow \infty \text{.}
\end{align*}
Here we first relax the condition on the times at which the connection and the growth event occur, resorting to the fact that $\sigma_t$ is monotone increasing in $t$; then we use the independence and translation-invariance of the events $\{\lceil x \rceil + n \leftrightarrow S^{\varphi}_n(\lceil x \rceil + n) \cap \mb{C}^{\operatorname{u}} \text{ in } \sigma^{\operatorname{u}}_{t_c}\}$ and $\operatorname{G}_{t_c - n^{-3/4 + \delta}, t_c, \lceil x \rceil + n}$; in the next step we apply equation (\ref{eq half-plane one-arm}); finally we use the inequality $1 - e^{-y} \leq y$ which is valid for all $y \in \mb{R}$. Since $-13/12 + \delta < -1$ holds by our choice of $\delta$, the previous estimate shows
\begin{align*}
\sum_{n = 1}^{\infty} \mf{P} \left[ \mc{D}^{\varphi, \delta}_{x, n} \right] < \infty \text{.}
\end{align*}
Invoking the Borel-Cantelli lemma again, we get
\begin{align} \label{thm critical height of destruction eq limsup D}
\mf{P} \left[ \limsup_{n \rightarrow \infty} \mc{D}^{\varphi, \delta}_{x, n} \right] = 0 \text{.}
\end{align}
Together with (\ref{thm critical height of destruction eq A subset B}) and (\ref{thm critical height of destruction eq limsup B}), equation (\ref{thm critical height of destruction eq limsup D}) yields the proof of (\ref{thm critical height of destruction eq to prove}) and hence of Theorem~\ref{thm critical height of destruction}.

\paragraph{Acknowledgement.}
I am grateful to Franz Merkl for helpful discussions and remarks. This work was supported by a scholarship from the Studienstiftung des deutschen Volkes.

\bibliography{critical_heights_of_destruction}
\bibliographystyle{alpha}

\end{document}